\documentclass[10pt,a4paper]{amsart}
\usepackage{amssymb}
 \usepackage[latin1]{inputenc}
 \usepackage[english]{babel}
 \usepackage{amsfonts}
 \usepackage{amsmath}
 \usepackage{amsthm}
 \usepackage{fancyhdr}
 \usepackage{colortbl}
 \usepackage{color}
 \usepackage[pdftex]{hyperref}
\newtheorem{theorem}{Theorem}
\newtheorem{proposition}{Proposition}[section]
\newtheorem{corollary}[proposition]{Corollary}
\newtheorem{lemma}[proposition]{Lemma}
\newtheorem*{problem}{Problem}
\theoremstyle{remark}
\newtheorem{remark}[proposition]{Remark}

\theoremstyle{definition}

\newtheorem*{acknowledgements}{Acknowledgements}

\newcommand{\R}{\mathbb{R}}
\newcommand{\C}{\mathbb{C}}

\newcommand{\N}{\mathbb{N}}
\newcommand{\E}{\mathbb E}

\begin{document}


\title{The Calder\'on problem with corrupted data}

\author{Pedro Caro}
\author{Andoni Garcia}
\address{BCAM - Basque Center for Applied mathematics}
\email{pcaro@bcamath.org}
\email{agarcia@bcamath.org}

\begin{abstract}
We consider the inverse Calder\'on problem consisting of determining the
conductivity inside a medium by electrical measurements on its surface.
Ideally, these 
measurements determine the Dirichlet-to-Neumann map and, therefore, one usually 
assumes the data to be given by such map. This situation corresponds to having 
access to infinite-precision measurements, which is totally unrealistic. In this 
paper, we study the
Calder\'on problem assuming the data to contain measurement errors and provide 
formulas to reconstruct the conductivity and its normal derivative on the 
surface. Additionally, we state the rate convergence of the 
method. Our approach is theoretical and has a stochastic flavour.
\end{abstract}

\date{\today}


\maketitle

\section{introduction}

In 1980, Calder\'on \cite{zbMATH05684831} proposed the following inverse boundary 
value problem: Let $D$ be a bounded domain in $\R^d$ ($d \geq 2$) with Lipschitz 
boundary $\partial D$, and let $\gamma$ be a real bounded measurable function in 
$D$ with a positive lower bound $\gamma_0$. Consider 
the linear map $\Lambda_\gamma : H^{1/2} (\partial D) \to H^{-1/2} 
(\partial D)$ 
defined---in a weak sense---by
\[
\Lambda_\gamma f = \gamma \partial_\nu u|_{\partial D}
\]
where $\partial_\nu = \nu \cdot \nabla$, with $\nu$ denoting the outward unit 
normal vector to $\partial D$,
and $ u \in H^1 (D) $ is the solution of the boundary value problem
\begin{equation}
	\left\{
		\begin{aligned}
		\nabla \cdot (\gamma \nabla u)  &= 0 \enspace\  \text{in} \ D, \\
		u|_{\partial D} &= f.
		\end{aligned}
	\right.
	\label{pb:BVP}
\end{equation}
In the literature, $\Lambda_\gamma$ is referred as the Dirichlet-to-Neumann map 
associated to $\gamma$ (DN map for short). The inverse 
problem is to decide whether $\gamma$ is uniquely determined by $\Lambda_\gamma$, 
and to calculate $\gamma$ in terms of~$\Lambda_\gamma$ if $\gamma$ is indeed 
determined by $\Lambda_\gamma$.

This problem originates in electrical prospecting. If $D$ represents an 
inhomogeneous conductive medium with conductivity $\gamma$, the inverse Calder\'on 
problem is to determine the conductivity $\gamma$ in $D$ by means of steady 
state electrical measurements carried out on the surface of $D$. In this physical 
situation, $f$ represents the electric potential on the surface and 
$\Lambda_\gamma f$ represents the normal component of the outgoing electric 
current density on the surface. Ideally, $\Lambda_\gamma$ is determined 
through measurements effected on $\partial D$.

Implementing the theoretical results of the Calder\'on problem presents
several non-trivial challenges. This is because theoretically one assumes to have access
to infinite-precision measurements and to infinite many pieces of
data, corresponding to knowing the whole graph of the DN map.
Neither of these assumptions are justified in practice.
On the one hand, only a finite number of measurements can be made to
obtain our data. On the other hand, the  data obtained will be corrupted
by measurement errors and so they will not even lie on the graph of the DN map.
The objective of this paper will be to address the question of
data corruption in the Calder\'on problem. For this purpose, we assume data to be 
given by points on the graph of the DN map plus an error modelled by \textit{random white noise}.
In mathematical terms, we consider a complete probability
space $(\Omega, \mathcal{H}, \mathbb{P})$, and a countable family
$\{ X_\alpha : \alpha \in \N^2 \}$
of independent complex Gaussian random variables
$X_\alpha : \omega \in \Omega \mapsto X_\alpha (\omega) \in \C $ such that
\begin{equation}
\E X_\alpha = 0, \qquad \E (X_\alpha \overline{X_\alpha}) = 1, \qquad 
\E(X_\alpha X_\alpha) = 0 \qquad \forall \alpha \in \N^2.
\label{id:definingX_alpha}
\end{equation}
We adopt the standard notation for the expectation of a random variable $X$:
\[
\E X = \int_\Omega X \,d\mathbb{P}.
\]
Then, we propose to define the noisy data for the Calder\'on problem as the bilinear form
\begin{equation}
\mathcal{N}_\gamma (f, g) = \int_{\partial D} \Lambda_\gamma f \, g \enspace + 
\sum_{\alpha \in \N^2} (f|e_{\alpha_1}) (g|e_{\alpha_2}) X_\alpha \qquad \forall 
f,g\in H^{1/2}(\partial D)
\label{id:noisy_data}
\end{equation}
where $\alpha = (\alpha_1, \alpha_2)$, $\{ e_n : n \in \N \}$ is an orthonormal 
basis of
$L^2 (\partial D)$ and $(\phi|\psi) = \int_{\partial D} \phi \overline{\psi}$.
\footnote{
The first integral on the definition of $\mathcal{N}_\gamma$ is an abuse of notation,
in fact meaning the duality pairing between $H^{1/2}(\partial D)$ and
$H^{-1/2}(\partial D)$.}  We will see,
in the corollary \ref{cor:noisy_data} below, that $\mathcal{N}_\gamma (f,g) \in 
L^2(\Omega, \mathcal{H}, \mathbb{P})$, and consequently $|\mathcal{N}_\gamma (f, g)| < 
\infty$ almost surely. Note that
\[\E \mathcal{N}_\gamma (f, g) = \int_{\partial D} \Lambda_\gamma f \, g, \]
which corresponds to saying that, with access to many independent outcomes
$\{ \mathcal{N}_\gamma (f, g){\scriptstyle( \omega_n )} : n \in \N \},$
we can filter out the noise by averaging
\[\frac{1}{N} \sum_{n=0}^{N - 1} \mathcal{N}_\gamma (f, g){\scriptstyle( \omega_n 
)} \xrightarrow[N \to \infty]{} \int_{\partial D} \Lambda_\gamma f \, g. \]
In practice, repetitions of the same measurement do not oscillate enough
to be filtered out by averaging. Therefore, our objective should avoid averaging
different realizations and show that
a single realization of $\mathcal{N}_\gamma (f, g)$ is enough to reconstruct 
$\gamma$.
\begin{problem}\label{pb:noisy} \sl Assuming $\gamma$ and $\partial D$ to be as 
smooth as needed,
 show that $\gamma$ can be calculated from~$\mathcal{N}_\gamma $
almost surely.
\end{problem}
We will see in the lemma \ref{lem:error} below that the error satisfies
\[\E \, \Big| \sum_{\alpha \in \N^2} (f|e_{\alpha_1}) (g|e_{\alpha_2}) X_\alpha 
\Big|^2 = \| f \|^2_{L^2(\partial D)} \| g \|^2_{L^2(\partial D)},\]
which means that the variance of the error depends on the inputs $f$ and $g$ used
to test the medium $D$. This model allows us to consider situations where the 
device used to obtain the boundary data decalibrates when the \textit{strength} of 
the electrical potential and the induced outgoing current increase.

The exact definition for the error has a purely theoretical motivation, and one 
could have replaced the space $L^2(\partial D)$ for other Hilbert spaces as 
$H^{1/2} (\partial D)$ or $H^{-1/2} (\partial D)$, however, the analysis carried 
out in this paper would be different. Note that in the case of $L^2$ the 
covariance operator (associated to the error) would be the identity---zeroth order 
operator, while in the cases $H^{1/2}$ and $H^{-1/2}$ would corresponds to 
operators of order $1$ and $-1$ respectively.

Saying that our error is modelled by a random white noise may seem vague and 
imprecise, but we hope it is not confusing. To clarify this comment, note that, 
given $m \in \N$, the linear map
\[W_m : g \longmapsto \sum_{n \in \N} (g|e_n) X_{(m,n)}\]
corresponds to a typical white noise. Thus, the error in our model is representing 
the mapping
\[f \longmapsto \sum_{m \in \N} (f | e_m) W_m.\]

The question of how to model the noise in inverse problems is of capital 
importance,
since infinite-precision measurements are totally unrealistic.
There seem to be two different approaches: one based on deterministic 
regularization
techniques, assuming the noise to be deterministic and small \cite{zbMATH03100841, zbMATH03227378}; and  
another
based on a statistical point of view \cite{zbMATH03214847, zbMATH03315672},
which does not need to assume smallness of the noise. See also the works \cite{zbMATH06298490, zbMATH06653178}.
Knudsen, Lassas, Mueller and 
Siltanen
\cite{zbMATH05679626} used regularization techniques to study the 
Calder\'on problem in 
dimension
$d = 2$ with noisy data. In order to carry out their deterministic 
analysis, they assumed
the noise level to be small. Our approach has a stochastic flavour with no 
restriction on the size of
the noise. In the context of the Calder\'on problem, this seems to 
be a new
approach. In this paper we show that $\gamma|_{\partial D}$ and 
$\partial_\nu \gamma|_{\partial D}$ can be reconstructed from a single 
realization of $\mathcal{N}_\gamma$.

\begin{theorem}\label{th:gamma}\sl Let $D$ be a bounded domain of $\R^d$ 
($d \geq 2$) with Lipschitz boundary $\partial D$. Consider $\gamma$ a 
Lipschitz continuous conductivity in $\overline{D}$. Then, for almost 
every $P \in  \partial D $, there exists an explicit sequence $\{ f_N : N 
\in \N \setminus\{0\} \}$ in $H^{1/2}(\partial D)$ such that
\[\lim_{N \to \infty} \mathcal{N}_\gamma (f_N, \overline{f_N}) = \gamma(P) 
\]
almost surely.
\end{theorem}

Our theorem only establishes a reconstruction procedure for almost every 
$P$ in $\partial D$. However, in the proposition \ref{prop:recon_gamma} we 
describe the set of boundary points for which the reconstruction algorithm 
works. It is worth to point out that this description only requires local 
smoothness of $\partial D$. In fact, if the domain was $C^1$ the theorem 
would hold for every point $P \in \partial D$.

The theorem \ref{th:gamma} extends a result with ideal data due to Brown 
\cite{zbMATH01731190} for the particular case that $\gamma$ is Lipschitz---Brown's 
theorem holds for very low regular conductivities. We believe that our theorem also holds 
at the same level of regularity with no extra effort.

The rate convergence of the limit in the theorem \ref{th:gamma} is described in 
the next theorem.

\begin{theorem}\label{th:stability_gamma}\sl Let $D$ be a bounded domain 
of $\R^d$ ($d \geq 2$) with a $C^{1,\theta}$ boundary $\partial D$ for 
$0 < \theta < 1$. Consider $\gamma$ as in the theorem \ref{th:gamma}.
Then, for every $P \in \partial D$, there exist an explicit sequence $\{ 
f_N : N \in \N\setminus\{0\} \}$ in $H^{1/2}(\partial D)$ and a constant 
$C > 0$ (depending on $d$, $\partial D$, a lower bound on $\gamma_0$ and 
an upper bound for $\| \gamma \|_{C^{0,1}(\overline{D})}$) such that, for 
every $ \epsilon > 0$, we have
\[ \mathbb{P} \{ | \mathcal{N}_\gamma (f_N, \overline{f_N}) - 
\gamma(P) | \leq C N^{-\theta/(1 + \theta)} \} \geq 1 - \epsilon \qquad 
\forall \, N \geq c \epsilon^{-\frac{1+\theta}{1 - \theta}}. \]
Here $c$ only depends on $\partial D$ and $\theta$.
\end{theorem}


In this theorem, the regularity of $\gamma$ could have been lowered to $C^{0,
\theta}$ with no extra effort and no loss on the rate of convergence. However, in 
order to get a rate of convergence of the type stated in our theorem, the 
method requires H\"older continuity for the 
conductivity and the first derivatives of the functions describing locally the boundary of $D$. We 
believe that this a priori regularity is also required when having ideal data. 
However, the stability of the problem for ideal data is Lipschitz under the 
assumptions of Brown's theorem \cite{0266-5611-32-11-115015}. This seems to tell 
that even if a reconstruction 
method provides Lipschitz stability for the problem, the rate of convergence of the 
same method could be worse or require extra assumptions.

In the next theorem we provide a formula to reconstruct the normal derivative of the 
conductivity at the boundary, once we know the conductivity at the 
boundary. For this, we will use a reference medium with an homogeneous 
conductivity identically one. Its corresponding DN map will be denoted by 
$\Lambda$.

\begin{theorem}\label{th:part_gamma}\sl Let $D$ be a bounded domain of $\R^d$ 
($d \geq 2$) with $C^{1,1}$ boundary $\partial D$ and assume $\gamma \in C^{1,1}
(\overline{D})$. Then, for
every $P \in  \partial D $, there exists an explicit family $\{ f_t : t 
\geq 1 \}$ in $H^{1/2}(\partial D)$ such that
\[\lim_{N \to  \infty} \frac{1}{T_N}\int_{T_N}^{2T_N} \Big[ \mathcal{N}_\gamma (f_{t^2}, \overline{f_{t^2}}/\gamma) - 
\int_{\partial D} \Lambda f_{t^2} \,\overline{f_{t^2}} \, \Big] \, dt = \frac{ 
\partial_{\nu_P} \gamma(P) + i\tau_P \cdot \nabla\gamma(P)}{\gamma(P)} \]
almost surely where $N \in \N \setminus \{0\}$ and $T_N = N^{3+3\theta/2}$ with $\theta \in (0, 1)$. Here $\nu_P$ is the outward unit normal vector to $\partial D$ at $P$ and $\tau_P$ denotes any unitary tangential vector at $P$.
\end{theorem}

Brown and Salo \cite{zbMATH05045048} proved a similar result to ours for the
steady state heat equation with convection assuming access to infinite-precision data. Their result could
be rewritten, in the case of ideal data, for the conductivity equation 
assuming $\gamma$ and $\partial D$ to be $C^1$. Our theorem \ref{th:part_gamma} 
extends this for the particular case where $\gamma$ and $\partial D$ are 
$C^{1,1}$. In this case, our method fails for less regular assumptions on the boundary $\partial D$---see the 
lemma \ref{lem:filtering} below---however, one may expect our assumption on the regularity of the conductivity to be relaxed. In the appendix of \cite{0266-5611-32-11-115015}, 
Brown in collaboration with Garc\'ia and Zhang proved that the normal derivative of the 
conductivity on the boundary can be recovered from ideal data assuming the boundary to be 
Lipschitz. This approach does not seem to be so convenient for our case since the formula 
is non-linear with respect to the data (see the theorem 7 in 
\cite{0266-5611-32-11-115015}) and this may cause difficulties when filtering 
out the noise.

Our last theorem describes the rate of convergence of the limit in the previous 
theorem.

\begin{theorem}\label{th:stability_part_gamma}\sl Let $D$ be a bounded domain of 
$\R^d$ 
($d \geq 2$) with $C^{1,1}$ boundary $\partial D$ and assume $\gamma \in C^{1,1}
(\overline{D})$. Consider $P\in\partial D$ and $\{ f_t : t \geq 1 \}$ the family of the 
theorem \ref{th:part_gamma}. For every $N \in \N\setminus \{ 0 \}$, set 
\[Y_N = \frac{1}{T_N}\int_{T_N}^{2T_N} \Big[ 
\mathcal{N}_\gamma (f_{t^2}, \overline{f_{t^2}}/\gamma) - \int_{\partial D} 
\Lambda f_{t^2} \,\overline{f_{t^2}} \, \Big] \, dt \]
with $T_N = N^{3+3\theta/2}$ for $\theta \in (0, 1)$.
Then, there exists a constant 
$C > 0$ (depending on $d$, $\partial D$, a lower bound on $\gamma_0$ and 
an upper bound for $\| \gamma \|_{C^{1,1}(\overline{D})}$) such that, for 
every $ \epsilon > 0$, we have
\[\mathbb{P} \Big\{ \big| Y_N - \frac{ \partial_{\nu_P} \gamma(P) + i\tau_P \cdot 
\nabla\gamma(P)}{\gamma(P)} \big| \leq C N^{-\theta} \Big\} \geq 1 - \epsilon 
\qquad \forall \, N > c \epsilon^{-\frac{1}{1 - \theta}}. 
\]
Here $c$ depends on $\theta$, $d$, $\partial D$, a lower bound for $\gamma_0$ and an upper bound for $\| \gamma \|_{C^{1,1}(\overline{D})}$.
\end{theorem}

The noise have been assumed to be Gaussian, however, in this paper 
this is not required. The reason for us to define the noise as Gaussian is because 
we believe that this will be convenient for the reconstruction of $\gamma$ 
in the interior of $D$.

To prove these theorems, we use the family of solutions constructed by Brown and 
Salo in the papers \cite{zbMATH01731190,zbMATH05045048}. Our main contribution 
consists of noting that these solutions are robust enough to get rid of the 
measurement errors by making an appropriate averaging on the parameter of the 
family. For 
the theorems \ref{th:gamma} and \ref{th:stability_gamma} this is not even required 
because $\| f_N \|_{L^2(\partial D)} = \mathcal{O} (N^{-1/2})$. However, for the theorems 
\ref{th:part_gamma} and \ref{th:stability_part_gamma} we only have that $\| f_N \|
_{L^2(\partial D)} = \mathcal{O} (1)$, which makes necessary the average in 
$\sqrt{N}$. This could be thought as an ergodic property of the traces of this 
family when applied to the noisy data. In other words, the noisy data generated by 
this family of solutions is statistically stable. We hope this paper could 
inspire a different way of dealing with noise in the numerical reconstruction of 
the conductivity.

The classical references for the Calder\'on problem with full ideal data and 
isotropic conductivities are: the works \cite{zbMATH03939884, zbMATH03957806} where Kohn and 
Vogelius proved boundary identification and interior uniqueness of analytic and piecewise analytic conductivities, global uniqueness for $d \geq 3$ \cite{zbMATH04015323} by Sylvester and Uhlmann,  
the work of stability \cite{zbMATH03998383} due to Alessandrini, reconstruction by 
Nachman \cite{zbMATH04105476} and uniqueness in $d=2$ \cite{zbMATH00854849} due to 
Nachman. See also \cite{zbMATH04028037, zbMATH00004861}.
More recent references dealing with questions of regularity in dimension $d = 2$ are
\cite{zbMATH01044195, zbMATH05050053} for uniqueness and \cite{zbMATH01649244, 
zbMATH05223898, zbMATH05704414} for stability. For the uniqueness in dimension $d \geq 3$
\cite{zbMATH00912089, zbMATH02102106, zbMATH02005267, zbMATH06145493, zbMATH06490961, 
zbMATH06534426}, the stability \cite{zbMATH06117512} and the reconstruction 
\cite{0266-5611-32-11-115015}.

Numerical reconstruction on the boundary with infinite precision measurements have
been investigated for $d=2,3$ in \cite{zbMATH02215011, zbMATH05937731}. In 
collaboration with Luca Gerardo-Giorda and Mar\'ia Jes\'us Mu\~noz L\'opez,
we are implementing numerically this scheme of reconstruction with corrupted data.

Regarding stochastic approaches, Dunlop and Stuart have recently given a rigorous 
Bayesian formulation of the electrical impedance tomography problem \cite{zbMATH06648201}.

Our paper contains other two sections. In the first one, we prove the theorems 
\ref{th:gamma} and \ref{th:stability_gamma}. The second one is devoted to the theorems 
\ref{th:part_gamma} and \ref{th:stability_part_gamma}.

\section{Recovering the conductivity at the boundary}\label{sec:recoveryconduct}
In this section we prove the theorems \ref{th:gamma} and 
\ref{th:stability_gamma}. Here we assume $\gamma \in C^{0,1}
(\overline{D})$ with $\gamma (x) \geq \gamma_0 > 0$ for all $x\in 
\overline{D}$ and the boundary of $D$ to be represented locally by the 
graphs of some Lipschitz functions. Thus, for each $P\in\partial D$, there 
is a coordinate system $(y^{\prime}, y_d)\in \R^{d-1}\times\R$, a constant 
$\rho>0$ and a Lipschitz function $\phi : \R^{d-1}\to\R$ so that
$$
B(P,\rho)\cap D=B(P,\rho)\cap\{y\in\R^d : y_d>\phi(y^{\prime})\}
$$
and
$$
B(P,\rho)\cap\partial D=B(P,\rho)\cap\{y\in\R^d : y_d=\phi(y^{\prime})\}.
$$
Let $(p',\phi(p^\prime))$ denote the coordinates of $P$ in the 
corresponding system and $F$ the map $F(x)=(x^{\prime}+p^{\prime}, 
x_d+\phi(x^{\prime}+p^{\prime}))$. Let $\tilde{D}$ denote the pre-image of 
$D$ under $F$, that is $\tilde{D}=F^{-1}(D)$.

Before going further, we observe that if $u$ solves the problem 
\eqref{pb:BVP}, then 
$F^\ast u(x) =u(F(x))$ solves the equation $\nabla\cdot (A_\gamma(x)\nabla F^\ast 
u)=0$ in $\tilde{D}$, where
\begin{equation}\label{eq:A}
A_\gamma(x)=\gamma(F(x))\nabla F^{-1}(F(x))\nabla F^{-1}(F(x))^{t},
\end{equation}
where $\nabla F^{-1} (y)^t$ is the transpose of
\[
\nabla F^{-1} (y) = \left[
\begin{matrix}
 I_{d-1} & 0\\
  -\nabla \phi(y') & 1
\end{matrix}
\right]
\]
with $I_{d-1}$ the identity in $d - 1$. Furthermore, since the Jacobian $|\text{det}\,
\nabla F(x)| = 1$, we 
have that
\begin{equation}
\begin{aligned}
\int_{\partial D} \Lambda_\gamma f \, \overline{f} \enspace &=\int_D \gamma(y) |\nabla 
u(y)|^2 \, dy\\
&=\int_{\tilde{D}}\gamma(F(x))\nabla u (F(x))\cdot \nabla \overline{u}(F(x))|\text{det}\, 
\nabla F(x)| \, dx\\
&=\int_{\tilde{D}}\nabla F^\ast u(x)\cdot(A_\gamma(x) \nabla F^\ast \overline{u}(x)) \, dx.
\end{aligned}
\label{id:change}
\end{equation}

We let $\eta: \R\to[0,1]$ to be a smooth function which satisfies $\eta(t)=1$, $|
t|\leq 1/2$ and $\eta(t)=0$, $|t|\geq 1$. We choose $\xi\in\R^d \setminus \{ 0 \}$ a 
constant vector for which $\xi\cdot A_\gamma(0)\xi=e_d\cdot A_\gamma(0)e_d$ and $\xi\cdot 
A_\gamma(0)e_d=0$. Note that this choice makes $\xi = (\xi', \xi_d)$ satisfy $|\xi'| \neq 0$. For $N \geq M \geq 1$, we set
\begin{equation}\label{eq:v_N}
a_{M,N} (x) =\eta(M|x^{\prime}|)\eta(M x_d) e^{N(i\xi-e_d)\cdot 
x}=\chi(M x)e^{N(i\xi-e_d)\cdot x},
\end{equation}
where the function $\chi(x) = \eta(|x'|) \eta(x_d)$ has been introduced to simplify 
notation. The lemmas 1 and 2 in \cite{zbMATH01731190} can be written in our particular 
case as follows:

\begin{lemma}[Brown \cite{zbMATH01731190}]\label{lem:Br1}\sl
If $\gamma\in C^{0,1}(\overline{D})$, $A_\gamma$ is defined as in \eqref{eq:A} 
and $\nabla \phi (p')$ exists. Then we have 
\begin{align*}
\int_{\tilde{D}} \nabla a_{M,N}(x)\cdot & (A_\gamma(x)\nabla \overline{a_{M,N}}(x)) dx= 
N M^{1-d}\gamma(P)(1+|\nabla\phi(p^{\prime})|^2)\int_{\R^{d-1}}\eta(|
x^{\prime}|)^2 dx^{\prime}
\\
&+\mathcal{O}\Big(M^{2 - d}+ N M^{-d} +N\int_{|x^{\prime}|\leq M^{-1}}|
\nabla\phi(x^{\prime}+p^{\prime})-\nabla\phi(p^{\prime})|\, dx^{\prime}\Big).
\end{align*}
The constant implicit in $\mathcal{O}$ depends on $d$ and on upper bounds 
for $\| \gamma \|_{C^{0,1}(\overline{D})}$ and $\| \phi \|_{C^{0,1}(\R^{d 
- 1})}$.
\end{lemma}

\begin{lemma}[Brown \cite{zbMATH01731190}]\label{lem:Br2}\sl
Consider $a_{M,N}$ as in \eqref{eq:v_N} and let $w_{M,N}$ solve the boundary value problem
\begin{equation*}
\left\{
\begin{aligned}
\nabla\cdot(A_\gamma\nabla w_{M,N})&=-\nabla\cdot(A_\gamma\nabla a_{M,N})  \enspace\  \text{in} \ 
\tilde{D},\\
w_{M,N}|_{\partial\tilde{D}}&=0.
\end{aligned}
\right.
\end{equation*}
If $\nabla \phi (p')$ exists, then
\begin{equation*}
\|\nabla w_{M,N}\|_{L^2(\tilde{D})} \lesssim N^{-\frac{1}{2}} M^\frac{3-d}{2} + M^{\frac{1-d}{2}}+N^{\frac{1}
{2}}\Big(\int_{|x^{\prime}|\leq M^{-1}}|
\nabla\phi(x^{\prime}+p^{\prime})-\nabla\phi(p^{\prime})|^2 \,
dx^{\prime}\Big)^{\frac{1}{2}}
\end{equation*}
where the implicit constant depends on $d$, a lower bound on $\gamma_0$ and on 
upper bounds for $\| \gamma \|_{C^{0,1}(\overline{D})}$ and $\| \phi \|
_{C^{0,1}(\R^{d - 1})}$.
\end{lemma}

Following \cite{zbMATH01731190}, we choose $M = N^{1/2}$
and consider the function $u_N$ defined by
\begin{equation}
F^\ast u_N = N^{-\frac{1}{2}} M^\frac{d - 1}{2} C_P (a_{M,N} + w_{M,N} )
\label{term:u_N}
\end{equation}
with $C_P = \big( (1+|\nabla\phi(p^{\prime})|^2)\int_{\R^{d-1}}\eta(|
x^{\prime}|)^2 dx^{\prime} \big)^{-1/2} $ and $f_N = u_N|_{\partial D}$. Note that after 
\eqref{id:noisy_data}, \eqref{id:change}, the lemmas \ref{lem:Br1} and \ref{lem:Br2}, and the Cauchy--Schwarz inequality we 
have
\begin{align*}
\mathcal{N}_\gamma (f_N, \overline{f_N}) &= \int_{\tilde{D}}\nabla F^\ast 
u_N(x)\cdot(A_\gamma(x) \nabla F^\ast \overline{u_N}(x)) \, dx \enspace 
+ \sum_{\alpha \in \N^2} (f_N|e_{\alpha_1}) (\overline{f_N}|e_{\alpha_2}) 
X_\alpha \\
& = \gamma(P) + \sum_{\alpha \in \N^2} (f_N|e_{\alpha_1}) (\overline{f_N}|
e_{\alpha_2}) X_\alpha +\mathcal{O}\Big( N^{-\frac{1}{2}} + h(M) \Big),
\end{align*}
where 
\[h(M) = \Big( M^{d - 1} \int_{|x^{\prime}|\leq M^{-1}}|
\nabla\phi(x^{\prime}+p^{\prime})-\nabla\phi(p^{\prime})|^2 \,
dx^{\prime}\Big)^{\frac{1}{2}}. \]

In the following lines, we will show that $\mathcal{N}_\gamma (f_N,
\overline{f_N})$ tends to $\gamma(P)$ as $N$ goes to infinity. The last
term will vanish, in the limit, under appropriate assumptions on $\partial 
D$. To show that the second vanishes almost surely in the limit, we will use a 
very simple idea of Lebesgue spaces---see the lemma \ref{lem:convergence} below. Before
this let us make some comments about $\mathcal{N}_\gamma$.

\begin{lemma}\label{lem:error}\sl
There exists a complete probability
space $(\Omega, \mathcal{H}, \mathbb{P})$, and a countable family
$\{ X_\alpha : \alpha \in \N^2 \}$
of independent complex random variables satisfying \eqref{id:definingX_alpha}.
Moreover, for every $f, g \in L^2(\partial D)$ we have 
that
\[\E \, \Big| \sum_{\alpha \in \N^2} (f|e_{\alpha_1}) (g|e_{\alpha_2}) X_\alpha 
\Big|^2 = \| f \|^2_{L^2(\partial D)} \| g \|^2_{L^2(\partial D)}.\]
\end{lemma}
\begin{proof} The existence part is a consequence of for example Ionescu--Tulcea's theorem. The second part is a simple consequence of the independence of $\{ X_\alpha : \alpha \in \N^2 
\}$ and the facts $\E (X_\alpha \overline{X_\alpha}) = 1$ for all $\alpha \in 
\N^2$ and that $\{ e_n : n \in \N \}$ is an orthonormal basis of $L^2 (\partial 
D)$.
\end{proof}

\begin{corollary}\label{cor:noisy_data}\sl The corrupted data
$$
\mathcal{N}_\gamma: H^{1/2}(\partial D)\times H^{1/2}(\partial D)\rightarrow L^2(\Omega, \mathcal{H}, P)
$$
is bounded in the sense that, there exists a constant $C>0$ depending on $d$ and $\partial D$ such that
\[
\E \big| \mathcal{N}_\gamma (f, g) \big|^ 2 \leq C ( 1 + \|\gamma\|^2_{L^\infty (D)} ) \| f \|^2_{H^{1/2}(\partial D)} \| g \|^2_{H^{1/2}(\partial D)}
\]
for all $f,g \in H^{1/2}(\partial D)$. Consequently, $| \mathcal{N}_\gamma (f, g)| < \infty$ almost surely.
\end{corollary}

As a consequence of the lemma \ref{lem:error}
\[\E \, \Big| \sum_{\alpha \in \N^2} (f_N|e_{\alpha_1}) (\overline{f_N}|e_{\alpha_2}) 
X_\alpha \Big|^2 = \| f_N \|^4_{L^2(\partial D)}.\]
On the other hand, a simple computation shows that
\begin{equation}
\| f_N \|^2_{L^2(\partial D)} \leq C_{\partial D} N^{-1}
\label{eq:cons_boundary}
\end{equation}
where the constant $C_{\partial D} > 0$ only depends on an upper bound for $\| \phi \|_{C^{0,1}(\R^{d - 1})}$.

The rest of argument relies on the following lemma.

\begin{lemma}\sl\label{lem:convergence}
Let $(X, \Sigma, \mu)$ be a measure space and $\{f_n : n\in\N\}$ be a sequence in $L^p(X, 
\Sigma, \mu)$ and $f\in L^p(X,\Sigma, \mu)$ with $1 \leq p < \infty$ such that
$$
f_n\to f \quad \text{in}\  L^p(X,\Sigma, \mu),
$$
as $n\to\infty$. Assume that there exists a sequence $\{\lambda_n : n\in \N\}$ of positive 
real numbers such that $\lambda_n\to 0$ as $n\to\infty$ and
$$
\sum_{n\in\N}\frac{1}{\lambda_n^p}\int_X|f_n-f|^p d\mu <\infty.
$$
Then,
$$
f_n(x)\to f(x)
$$
as $n\to\infty$ for almost every $x\in X$.

Assume furthermore that $\mu(X)<\infty$. Then, for every $\epsilon>0$, there exists a 
$n_0\in \N$ such that 
$$
\mu(\{x\in X : |f_n(x)-f(x)|\leq \lambda_n\})\geq \mu(X)-\epsilon,
$$
for $n\geq n_0$.
\end{lemma}

\begin{proof}
The first part of this lemma follows from the fact that
\begin{equation}
\mu \Big(\bigcap_{n = 1}^\infty \bigcup_{k = n}^\infty E_k\Big) = 0
\label{id:liminf}
\end{equation}
for $E_n = \{ x \in X : |f_n(x)-f(x)| > \lambda_n \}$. Indeed, if $x 
\notin \cap_{n = 1}^\infty \cup_{k = n}^\infty E_k $ there exists a $n_x 
\in \N$ such that $x \notin \cup_{k = n}^\infty E_k$ for all $n \geq n_x$, 
which means that
\[|f_n(x) - f(x)| \leq \lambda_n, \quad \forall\, n \geq n_x.\]
The identity \eqref{id:liminf} holds as a consequence of the following 
inequalities
\[\mu \Big(\bigcap_{n = 1}^\infty \bigcup_{k = n}^\infty E_k\Big) \leq 
\sum_{k = n}^\infty \mu (E_k) \leq \sum_{k = n}^\infty \frac{1}
{\lambda_k^p} \int_X |f_k - f|^p \, d\mu.\]

To prove the second part, let $n_0$ be such that
\begin{equation}
\sum_{n = n_0}^\infty \frac{1}{\lambda_n^p} \int_X |f_n - f|^p \, d\mu 
\leq \epsilon.
\label{in:choice_eps}
\end{equation}
Then,
\[\mu\Big(X \setminus \bigcup_{n = n_0}^\infty E_n\Big) \geq \mu(X) - 
\sum_{n = n_0}^\infty \mu (E_n) \geq \mu(X) - \epsilon \]
where $ x \in X \setminus \cup_{n = n_0}^\infty E_n$ if and only if $|
f_n(x) - f(x)| \leq \lambda_n$ for all $n \geq n_0$.
\end{proof}

\begin{remark} \label{rem:choice_n0} Note that, given $\epsilon > 0$, the $n_0$ stated in the lemma \ref{lem:convergence} only has to satisfy \eqref{in:choice_eps}.
\end{remark}

Applying the first part of this lemma to $L^2(\Omega, \mathcal{H}, 
\mathbb{P})$, the 
sequence
$$ \sum_{\alpha \in \N^2} (f_N|e_{\alpha_1}) (\overline{f_N}|
e_{\alpha_2}) X_\alpha $$
and $\lambda_N = N^{-\theta}$ with $\theta \in 
(0, 1/2)$, we have the following proposition:

\begin{proposition}\label{prop:recon_gamma} \sl Let $P \in \partial D$ 
such that the corresponding 
boundary function $\phi$ satisfies
\begin{equation}
\lim_{r \to 0} \frac{1}{r^{d - 1}} \int_{|x^{\prime}|\leq 
r}| \nabla\phi(x^{\prime}+p^{\prime})-\nabla\phi(p^{\prime})|^2 \,
dx^{\prime} = 0.
\label{id:Lebesgue_lim}
\end{equation}
Then, if $N \in \N \setminus \{ 0\}$ we have that
\[\lim_{N \to \infty} \mathcal{N}_\gamma (f_N, \overline{f_N}) = \gamma(P) \]
almost surely.
\end{proposition}
It is well known that, for almost every $P \in \partial D$, its 
corresponding boundary functions $\phi$ satisfies \eqref{id:Lebesgue_lim}. 
Therefore, the theorem \ref{th:gamma} holds.

The theorem \ref{th:stability_gamma} will be a consequence of the 
following proposition.
\begin{proposition}\sl Let $P \in \partial D$ 
such that the corresponding 
boundary function $\phi$ satisfies
\begin{equation}
|\nabla\phi(x^{\prime}+p^{\prime})-\nabla\phi(p^{\prime})| \leq L |x^{\prime}|^\theta
\label{in:C1lambda}
\end{equation}
with $L > 0$ and $0 < \theta < 1 $. Consider $u_N$ as in \eqref{term:u_N} with $M = N^{1/(1 + \theta)}$ and $N \in \N \setminus \{ 0 \}$. Then, there exists a constant $C > 0$ (depending on $d$, a lower bound on $\gamma_0$ and on 
upper bounds for $\| \gamma \|_{C^{0,1}(\overline{D})}$, $\| \phi 
\|_{C^{0,1}(\R^{d - 1})}$ and $L$) such that, for every $\epsilon > 0$, we have
\[ \mathbb{P} \{ | \mathcal{N}_\gamma (f_N, \overline{f_N}) - \gamma(P) | \leq C N^{-\theta/(1 + \theta)} \} \geq 1 - \epsilon \qquad \forall \, N \geq c \epsilon^{-\frac{1+\theta}{1 - \theta}}. \]
Here $c$ only depends on $C_{\partial D}$ and $\theta$.
\end{proposition}

\begin{proof} Using \eqref{id:noisy_data}, \eqref{id:change}, the lemmas \ref{lem:Br1} and \ref{lem:Br2}, and \eqref{in:C1lambda} we can see that
\begin{align*}
\mathcal{N}_\gamma (f_N, \overline{f_N}) &= \int_{\tilde{D}}\nabla F^\ast 
u_N(x)\cdot(A_\gamma(x) \nabla F^\ast \overline{u_N}(x)) \, dx \enspace 
+ \sum_{\alpha \in \N^2} (f_N|e_{\alpha_1}) (\overline{f_N}|e_{\alpha_2}) 
X_\alpha \\
& = \gamma(P) + \sum_{\alpha \in \N^2} (f_N|e_{\alpha_1}) (\overline{f_N}|
e_{\alpha_2}) X_\alpha +\mathcal{O}\Big( N^{-\frac{1}{2}} + M^{-\theta} \Big).
\end{align*}
Applying the second part of the lemma \ref{lem:convergence} for 
$L^2(\Omega, \mathcal{H}, 
\mathbb{P})$, $\lambda_N = N^{-\theta/(1 + \theta)}$ and the 
sequence $ \{ \sum (f_N|e_{\alpha_1}) (\overline{f_N}|
e_{\alpha_2}) X_\alpha : N \in \N \setminus \{ 0 \} \} $, and using
\[\E \, \Big| \sum_{\alpha \in \N^2} (f_N|e_{\alpha_1}) (\overline{f_N}|e_{\alpha_2}) 
X_\alpha \Big|^2 \leq C_{\partial D}^2 N^{-2}\]
with $C_{\partial D}$ as in \eqref{eq:cons_boundary},
we know that:
\[\mathbb{P} \{ | \sum (f_N|e_{\alpha_1}) (\overline{f_N}|
e_{\alpha_2}) X_\alpha | \leq N^{-\theta/(1 + \theta)} \}  \geq 1 - \epsilon\]
for every $N \geq N_0 $. According to the remark \ref{rem:choice_n0}, it is enough to choose $N_0 > 1$ satisfying
\[\sum_{N = N_0}^\infty \frac{C_{\partial D}^2}{N^\frac{2}{1 + \theta}} \leq \epsilon,\]
which holds whenever
\begin{equation}
\frac{C^2_{\partial D}}{ \epsilon} \frac{1+ \theta}{1 - \theta} < (N_0 - 1)^\frac{1 - \theta}{1 + \theta}.
\label{es:small_c}
\end{equation}
From the identity at the beginning of this proof and the choice $M = N^{1/(1+\theta)}$, we see that there exists a constant $C>0$ such that
\[\{ | \sum (f_N|e_{\alpha_1}) (\overline{f_N}|
e_{\alpha_2}) X_\alpha | \leq N^{-\theta/(1 + \theta)} \} \subset \{ | 
\mathcal{N}_\gamma (f_N, \overline{f_N}) - \gamma(P) | \leq C  
N^{-\theta/(1 + \theta)} \},\]
which is enough to conclude the proof.
\end{proof}

If $D$ has a $C^{1,\theta}$ boundary, then every point $P \in \partial D$ satisfies \eqref{in:C1lambda}, and consequently the theorem \ref{th:stability_gamma} holds.

\section{The normal derivative of the conductivity at the boundary}
Here we prove the theorems \ref{th:part_gamma} and \ref{th:stability_part_gamma}. We start by considering an integral identity that brings up the gradient of the conductivity.

\begin{lemma}\sl
Let $u\in H^1(D)$ be the unique solution of the boundary value problem \eqref{pb:BVP} and $v\in H^1(D)$ the harmonic extension in $D$ of $g \in H^{1/2}(\partial D)$. Then,
\begin{equation}\label{eq:idenderiv}
\int_{\partial D} \Big( \frac{1}{\gamma}\Lambda_\gamma-\Lambda \Big)f \,  g \enspace = \, -\int_D \frac{\nabla\gamma}{\gamma}\cdot\nabla u \, v
\end{equation}
where $\Lambda$ is the $DN$ map associated to the conductivity identically one.
\end{lemma}

\begin{proof}
By the definition of the DN map,
\[\int_{\partial D} \frac{1}{\gamma}\Lambda_\gamma f \,  g \enspace = \, 
\int_D \gamma \nabla u \cdot \nabla \Big( \frac{v}{\gamma} \Big) \]
since $v/\gamma \in H^1(D)$ and $v|_{\partial D} = g$. Furthermore, we 
have
\[\int_D \gamma \nabla u \cdot \nabla \Big( \frac{v}{\gamma} \Big) 
\enspace = \int_D \nabla u \cdot \nabla v \, -\int_D \frac{\nabla\gamma}
{\gamma}\cdot\nabla u \, v  \]
by the Leibniz rule.
Consider $w \in H^1(D)$ the harmonic extension in $D$ of $f$. Adding and 
subtracting $\nabla w$ as appropriate we see that
\[\int_D \nabla u \cdot \nabla v \enspace = \, \int_D \nabla (u - w) \cdot 
\nabla v \, + \, \int_{\partial D} \Lambda f \, g \]
by the definition of $\Lambda$. Since $v$ is harmonic in $D$, the first 
term in the right hand side of the previous identity vanishes. Eventually, 
the integral identity we want to prove follows from the former 
considerations.
\end{proof}

With this identity at hand, we just plug in $a_{M,N}$ as in \eqref{eq:v_N} in order to obtain an asymptotic equality similar to the one in the lemma \ref{lem:Br1}.

\begin{lemma} \label{lem:leadingDEV}\sl Assume $\gamma\in C^{1,1}(\overline{D})$ and $\phi \in C^{1,1}(\R^{d - 1})$. Then we have
\begin{align*}
\int_{\tilde{D}} &\frac{\nabla\gamma(F(x))}{\gamma(F(x))}\cdot \big( 
\nabla F^{-1}(F(x))^t \nabla a_{M,N}(x) \big) \overline{a_{M,N}(x)} \, dx \\
& = M^{1-d} \frac{1}{2} \frac{\nabla\gamma(P)}{\gamma(P)} \cdot 
\big( \nabla F^{-1}(P)^t (i\xi - e_d) \big) \int_{\R^{d-1}}\eta(|
x^{\prime}|)^2\, dx^{\prime} \\
&\quad  +\mathcal{O}\Big(M^{2 - d} N^{-1} + M^{-d} \Big).
\end{align*}
The constant implicit in $\mathcal{O}$ depends on $d$, a lower bound on $\gamma_0$ and on upper bounds for $\| \gamma \|_{C^{1,1}(\overline{D})}$ and $\| \phi \|_{C^{1,1}(\R^{d - 1})}$.
\end{lemma}
\begin{proof} By the definition of $a_{M,N}$, the term to be computed equals
\begin{align*}
N \int_{\tilde{D}} & \frac{\nabla\gamma(F(x))}{\gamma(F(x))}\cdot \big( 
\nabla F^{-1}(F(x))^{t}(i\xi-e_d) \big) \chi(M x)^2 e^{-2Nx_d} dx\\
& + M\int_{\tilde{D}}\frac{\nabla\gamma(F(x))}
{\gamma(F(x))}\cdot \big( \nabla F^{-1}
(F(x))^{t}\nabla\chi(Mx) \big) \chi(M x) e^{-2Nx_d} dx.
\end{align*}
The last of these two addends is $\mathcal{O} (M^{2-d} N^{-1})$. The first of them is analysed according to the following decomposition
\begin{align*}
\nabla F^{-1}(F(x)) \frac{\nabla\gamma(F(x))}{\gamma(F(x))} &= \nabla 
F^{-1}(P) \frac{\nabla\gamma(P)}{\gamma(P)} + \big( \nabla F^{-1}(F(x)) - 
\nabla F^{-1}(P) \big) \frac{\nabla\gamma(P)}{\gamma(P)} \\
& \quad + \nabla F^{-1}(F(x)) \Big( \frac{\nabla\gamma(F(x))}{\gamma(F(x))} - \frac{\nabla\gamma(P)}{\gamma(P)} \Big).
\end{align*}
The first term yields
\[  N \frac{\nabla\gamma(P)}{\gamma(P)} \cdot \big( \nabla F^{-1}(P)^t 
(i\xi - e_d) \big)\int_{\tilde{D}} \chi(M x)^2 e^{-2Nx_d} dx, \]
which is easily computed by using that
\begin{align*}
\int_{\tilde{D}} \chi(M &x)^2 e^{-2Nx_d} dx \\
& = M^{1 - d} \int_{\R^{d-1}}\eta(|x^{\prime}|)^2\, dx^{\prime} \Big( \frac{1}{2N} + \int_0^\infty (\eta(Mx_d) - 1) e^{-2Nx_d} \,dx_d  \Big).
\end{align*}
This already provides the leading term in the asymptotic identity stated in the lemma. We are now left with the second and third terms on the previous decomposition. For the second of them, we just need to use that 
\[\big| \nabla F^{-1}(F(x)) - \nabla F^{-1}(F(0)) \big| \lesssim |
\nabla\phi(x^{\prime}+p^{\prime})-\nabla\phi(p^{\prime})|, \]
which yields a term of the order
$$\mathcal{O}\Big(\int_{|x^{\prime}|\leq M^{-1}}| 
\nabla\phi(x^{\prime}+p^{\prime})-\nabla\phi(p^{\prime})|\, 
dx^{\prime}\Big) = \mathcal{O}(M^{-d}).$$
Eventually, for the third term arising in the decomposition, we use that
\[|\nabla \gamma (F(x)) - \nabla \gamma (F(0))| \lesssim |x|,\]
which yields a term of the order $\mathcal{O}(M^{-d})$. This ends the proof of this lemma.
\end{proof}

Now we need a lemma similar to the \ref{lem:Br2} but for harmonic functions. In this case, the result was proved by Brown and Salo (the lemma 2.5 in \cite{zbMATH05045048}):

\begin{lemma}[Brown--Salo \cite{zbMATH05045048}]\label{lem:BrSa}\sl Assume  
$\phi \in C^{1,1}(\R^{d - 1})$. 
Consider $a_{M,N}$ as in \eqref{eq:v_N} with $M = N^{1/2}$. Let 
$z_{M,N}$ solve the boundary value problem
\begin{equation*}
\left\{
\begin{aligned}
\nabla\cdot(A\nabla z_{M,N})&=-\nabla\cdot(A\nabla a_{M,N})  \enspace\  \text{in} \ 
\tilde{D},\\
z_{M,N}|_{\partial\tilde{D}}&=0,
\end{aligned}
\right.
\end{equation*}
where $A(x)=\nabla F^{-1}(F(x))\nabla F^{-1}(F(x))^{t}$.
Then
\begin{equation*}
\| z_{M,N}\|_{L^2(\tilde{D})} \lesssim N^{-\frac{1}{2}} M^\frac{1-d}{2} 
\end{equation*}
where the implicit constant depends on $d$ and on 
upper bounds for $\| \phi \|_{C^{1,1}(\R^{d - 1})}$.
\end{lemma}

For the choice $M = N^{1/2}$,
we consider the functions $u_N$ and $v_N$ defined by
\begin{align*}
F^\ast u_N &= M^\frac{d - 1}{2} C^\prime_P (a_{M,N} + w_{M,N} )\\
F^\ast v_N &= M^\frac{d - 1}{2} C^\prime_P (a_{M,N} + z_{M,N} )
\end{align*}
with $C^\prime_P = \sqrt{2} (1+|\nabla\phi(p^{\prime})|^2)^{-1/4} \big(\int_{\R^{d-1}}\eta(|
x^{\prime}|)^2 dx^{\prime} \big)^{-1/2}$. Let $f_N$ denote $u_N|_{\partial 
D} = v_N|_{\partial D}$ and plug them in the right hand side of 
\eqref{eq:idenderiv}
\begin{equation*}
\int_D \frac{\nabla\gamma}{\gamma}\cdot\nabla u_N \, \overline{v_N} \enspace = \int_{\tilde{D}} \frac{\nabla\gamma (F(x))}{\gamma(F(x))}\cdot \big( \nabla F^{-1}(F(x))^t \nabla F^\ast u_N (x) \big) \overline{F^\ast v_N (x)} \, dx.
\end{equation*}
By the identity \eqref{eq:idenderiv} and the lemma \ref{lem:leadingDEV}, we have that
\begin{align*}
\int_{\partial D} \Big( \frac{1}{\gamma}\Lambda_\gamma-\Lambda \Big)f_N \,  \overline{f_N} \enspace &= - \frac{\nabla\gamma(P)}{\gamma(P)} \cdot 
\frac{\nabla F^{-1}(P)^t (i\xi - e_d)}{(1+|\nabla\phi(p^{\prime})|^2)^{1/2}} + \mathcal{O} (M^{-1}) \\
& \quad + \mathcal{O} \Big( M^{d - 1} \| \delta \nabla a_{M,N} \|_{L^2(\tilde{D})} \| z_{M,N} / \delta \|_{L^2(\tilde{D})} \Big) \\
& \quad + \mathcal{O} \Big( M^{d - 1} \| \nabla w_{M,N} \|_{L^2(\tilde{D})} \| a_{M,N} + z_{M,N} \|_{L^2(\tilde{D})} \Big),
\end{align*}
where $\delta (x) $ denotes the distance between $x$ and $\partial \tilde{D}$. As Brown did in \cite{zbMATH01731190}, we use Hardy's inequality to bound
\[ \| z_{M,N} / \delta \|_{L^2(\tilde{D})} \lesssim \| \nabla z_{M,N} \|_{L^2(\tilde{D})}. \]
The terms $\| \nabla z_{M,N} \|_{L^2(\tilde{D})}$ and $\| \nabla w_{M,N} \|_{L^2(\tilde{D})}$ can be bounded according to the lemma \ref{lem:Br2}. For $\| z_{M,N} \|_{L^2(\tilde{D})}$ we will use the lemma \ref{lem:BrSa}. The remaining terms will be bounded as follows:
\begin{lemma} \sl
Let the function
$\phi \in C^{1,1}(\R^{d - 1})$ and $M = N^{1/2}$. Then, we have that
$\| a_{M,N} \|_{L^2(\tilde{D})} = \mathcal{O} (M^{(1 - d)/2} N^{-1/2})$ and $\| \delta \nabla a_{M,N} \|_{L^2(\tilde{D})} = \mathcal{O} (M^{(1 - d)/2} N^{-1/2})$.
\end{lemma}
\begin{proof}
We only consider $\| \delta \nabla a_{M,N} \|_{L^2(\tilde{D})}$, the other is a straightforward computation. It is enough to note that
\begin{align*}
\| \delta \nabla a_{M,N} \|_{L^2(\tilde{D})} &\lesssim N \Bigg( \int_{\tilde{D}} \chi(Mx)^2 e^{-2Nx_d} x_d^2 \, dx  \Bigg)^{1/2} \\
& \quad + M \Bigg( \int_{\tilde{D}} | \nabla \chi(Mx)| \chi(Mx) e^{-2Nx_d} x_d^2 \, dx  \Bigg)^{1/2}
\end{align*}
and estimate the first of these integrals, which is the one of highest order.
\end{proof}
These considerations, together with \eqref{id:noisy_data}, yield the 
asymptotic equality
\begin{equation}
\begin{aligned}
\mathcal{N}_\gamma (f_N, \overline{f_N}/\gamma) - \int_{\partial D} 
\Lambda f_N \,\overline{f_N} \, &= - \frac{\nabla\gamma(P)}{\gamma(P)} 
\cdot 
\frac{\nabla F^{-1}(P)^t (i\xi - e_d)}{(1+|\nabla\phi(p^{\prime})|
^2)^{1/2}} \\
&\quad + \sum_{\alpha \in \N^2} (f_N|e_{\alpha_1}) (\overline{f_N}/ 
\gamma|e_{\alpha_2}) X_\alpha + \mathcal{O} (N^{-\frac{1}{2}}).
\end{aligned}
\label{id:DEVwithNOISE}
\end{equation}

As in the section \ref{sec:recoveryconduct}, we need to filter out the 
noise in \eqref{id:DEVwithNOISE} to be able to recover $\partial_\nu 
\gamma|_{\partial D}$. The situation here is a bit more involved, since
\[\| f_N \|_{L^2(\partial D)} = \mathcal{O} (1).\]
However, averaging in the parameter $M = N^{1/2}$ we are able to get rid of the noise.
\begin{lemma}\label{lem:filtering}\sl We have that, for $T > 0$,
\[\mathbb{E} \bigg|\frac{1}{T}\int_T^{2T} \sum_{\alpha \in \N^2} (f_{t^2}|e_{\alpha_1}) 
(\overline{f_{t^2}}/ \gamma|e_{\alpha_2}) X_\alpha \, dt \bigg|^2 \leq \frac{C}{T^{2/3}}.\]
The constant $C > 0$ depends on $d$, a lower bound for $\gamma_0$ and upper bounds for $\| \gamma \|_{C^{1,1}(\overline{D})}$ and  $\| \phi \|_{C^{1,1}(\R^{d - 1})}$.
\end{lemma}
\begin{proof}
Start by noting that
\begin{align*}
\mathbb{E} \bigg|\frac{1}{T}\int_T^{2T} \sum_{\alpha \in \N^2} (f_{t^2}|e_{\alpha_1}) 
(\overline{f_{t^2}}/ \gamma|e_{\alpha_2}) X_\alpha \, dt \bigg|^2
= \frac{1}{T^2}\int_{Q_T} (f_{t^2}|f_{s^2}) 
(f_{s^2}/ \gamma|f_{t^2}/ \gamma) \, d(t,s)
\end{align*}
where $Q_T = [T,2T] \times [T, 2T]$. Consider $S \in (0 , T/2)$ to be 
chosen later and set
\begin{align*}
D(S) &= \{ (t,s) \in Q_T : t - S \leq s \leq t + S \}, \\
L(S) &= \{ (t,s) \in Q_T : T \leq s < t - S \}, \\
R(S) &= \{ (t,s) \in Q_T : t + S < s \leq 2T \}.
\end{align*}
A direct computation shows that
$|(f_{t^2}|f_{s^2})| + |(f_{s^2}/ \gamma|f_{t^2}/ \gamma)| \lesssim 1$, and consequently,
\[\frac{1}{T^2}\int_{D(S)} (f_{t^2}|f_{s^2}) 
(f_{s^2}/ \gamma|f_{t^2}/ \gamma) \, d(t,s) \lesssim \frac{S}{T}.\]
On the other hand, using that
\begin{equation}
(-i\xi' \cdot \nabla) e^{i(t^2 - s^2) \xi' \cdot x'} = |\xi'|^2 (t^2 - s^2) e^{i(t^2 - s^2) \xi' \cdot x'},
\label{id:oscillations}
\end{equation}
where $\xi = (\xi', \xi_d)$, integrating by parts and using the regularity for $\phi$ and $\gamma$ we can see that, whenever $t \neq s$,
\begin{equation}
|(f_{t^2}|f_{s^2})| + |(f_{s^2}/ \gamma|f_{t^2}/ \gamma)| \lesssim \frac{t + s + 1}{|t^2 
- s^2|}.
\label{es:oscillations+int_by_parts}
\end{equation}
Now we show how to perform the integration by parts for $|(f_{t^2}|f_{s^2})|$---the same argument is valid for $|(f_{s^2}/ \gamma|f_{t^2}/ \gamma)|$:
\begin{align*}
|(f_{t^2}|f_{s^2})| &= (C'_P)^2 t^\frac{d - 1}{2} s^\frac{d - 1}{2} \Big| \int_{\R^{d- 1}} \eta(t|x'|) \eta(s|x'|) e^{i(t^2 - s^2)\xi' \cdot x'} (1 + |\nabla \phi (x')|)^{1/2} \, dx' \Big| \\
& \lesssim \frac{t^\frac{d - 1}{2} s^\frac{d - 1}{2}}{|t^2 - s^2|} \int_{\R^{d- 1}} \big| \nabla \big( \eta(t|x'|) \eta(s|x'|) (1 + |\nabla \phi (x')|)^{1/2} \big) \big| \, dx'.
\end{align*}
In the last inequality, we have used identity \eqref{id:oscillations} and integrated by parts. The fact that $|\xi'| \neq 0$ is required justify the integration by parts. Eventually, to obtain \eqref{es:oscillations+int_by_parts} we just apply Leibniz rule and H\"older's inequality. As a consequence of \eqref{es:oscillations+int_by_parts}, we can bound
\[\frac{1}{T^2}\int_{R(S)} (f_{t^2}|f_{s^2}) 
(f_{s^2}/ \gamma|f_{t^2}/ \gamma) \, d(t,s) \lesssim \frac{1}{S^2}, \]
where we have used that the area of $R(S)$ is $\mathcal{O}(T^2)$ and that if $(t , s) \in R(S)$, then
$$s^2 - t^2 > (t+S)^2 - t^2 = 2St + S^2 > St \geq ST.$$
The same bounds hold for the integration on $L(S)$. Finally, we choose $S$ to satisfy $S T^{-1} = S^{-2}$ and get bound claim in the statement.
\end{proof}

As a consequence of the first part of the lemma \ref{lem:convergence} (with $\lambda_N = N^{-\theta}$ and $\theta \in (0, 1)$) we have that
\begin{equation}
\frac{1}{T_N}\int_{T_N}^{2T_N} \sum_{\alpha \in \N^2} (f_{t^2}|e_{\alpha_1}) 
(\overline{f_{t^2}}/ \gamma|e_{\alpha_2}) X_\alpha \, dt \longrightarrow 0
\label{lim:averaging}
\end{equation}
almost surely as $N \in \N \setminus \{ 0 \}$. Recall that $T_N = N^{3+3\theta/2}$.

\begin{proof}[Proof of the theorem \ref{th:part_gamma}]
Consider $\nu_P$ the unit normal vector to $\partial D$ at $P$ pointing outward and $\tau_P$ any unitary tangential vector at $P$. Let $\xi$ satisfy
\[ \tau_P = - \frac{\nabla F^{-1} (P)^t \xi}{(1 + |\nabla \phi (p')|^2)^{1/2}}.\]
Since 
\[ \nu_P = \frac{\nabla F^{-1} (P)^t e_d}{(1 + |\nabla \phi (p')|^2)^{1/2}}\]
we know that
$\xi\cdot A_\gamma(0)\xi=e_d\cdot A_\gamma(0)e_d$, $\xi\cdot 
A_\gamma(0)e_d=0$ and $\xi = (\xi', \xi_d)$ satisfies $\xi' \neq 0$. In this case, \eqref{id:DEVwithNOISE}
\begin{align*}
\mathcal{N}_\gamma (f_{t^2}, \overline{f_{t^2}}/\gamma) - \int_{\partial D} 
\Lambda f_{t^2} \,\overline{f_{t^2}} \, &= \frac{ \partial_{\nu_P} \gamma(P) + i\tau_P \cdot \nabla\gamma(P)}{\gamma(P)} \\
&\quad + \sum_{\alpha \in \N^2} (f_{t^2}|e_{\alpha_1}) (\overline{f_{t^2}}/ 
\gamma|e_{\alpha_2}) X_\alpha + \mathcal{O} (t^{-1}).
\end{align*}
The proof of the theorem ends just taking average in the interval $(T_N, 2T_N)$ for every 
term of the previous asymptotic identity and using \eqref{lim:averaging}.
\end{proof}

\begin{proof}[Proof of the theorem \ref{th:stability_part_gamma}]
Noting that
\[\mathbb{E} \bigg|\frac{1}{T_N}\int_{T_N}^{2T_N} \sum_{\alpha \in \N^2} (f_{t^2}|e_{\alpha_1}) 
(\overline{f_{t^2}}/ \gamma|e_{\alpha_2}) X_\alpha \, dt \bigg|^2 \leq \frac{C}{N^{2 + \theta}}\]
with $C$ as in the lemma \ref{lem:filtering},
we apply the second part of the lemma \ref{lem:convergence} for the sequence of random variables
\[\Big\{ \frac{1}{T_N}\int_{T_N}^{2T_N} \sum_{\alpha \in \N^2} (f_{t^2}|e_{\alpha_1}) 
(\overline{f_{t^2}}/ \gamma|e_{\alpha_2}) X_\alpha \, dt : N \in \N\setminus\{ 0 \} \Big\}\]
and $\lambda_N = N^{-\theta}$. Thus,
we have that
\[\mathbb{P} \Big\{ \big| \frac{1}{T_N}\int_{T_N}^{2T_N} \sum_{\alpha \in \N^2} (f_{t^2}|e_{\alpha_1}) 
(\overline{f_{t^2}}/ \gamma|e_{\alpha_2}) X_\alpha \, dt \big| \leq N^{-\theta} \Big\} \geq 1 - \epsilon\]
for all $N \geq N_0$. According to \ref{rem:choice_n0}, it is enough to choose $N_0$ satisfying
\[\sum_{N=N_0}^\infty \frac{C}{N^{2 - \theta}} < \epsilon \]
with $C$ as in the lemma \ref{lem:filtering}, which holds whenever
\[(N_0 - 1)^{1 - \theta} > \frac{C}{1 - \theta} \epsilon^{-1}.\]
Since
\begin{align*}
\Big\{ \big| \frac{1}{T_N}\int_{T_N}^{2T_N} \sum_{\alpha \in \N^2} &(f_{t^2}|e_{\alpha_1}) 
(\overline{f_{t^2}}/ \gamma|e_{\alpha_2}) X_\alpha \, dt \big| \leq N^{-\theta} \Big\} \\
& \subset \Big\{ \big| Y_N - \frac{ \partial_{\nu_P} \gamma(P) + i\tau_P \cdot \nabla\gamma(P)}{\gamma(P)} \big| \leq C N^{-\theta} \Big\},
\end{align*}
we can conclude the inequality stated in the theorem.
\end{proof}

\begin{acknowledgements} The authors are partially supported by BERC 2014-2017 and the MINECO grant BCAM Severo Ochoa SEV-2013-0323. PC is also supported by the MINECO project MTM2015-69992-R, and would like to thank Ikerbasque - Basque Foundation for Science for their support and encouragement. AG is also supported by the MINECO project MTM2014-53145-P. Finally, we would like to thank the comments and recommendations of the anonymous referees, and also the careful reading of our manuscript.
\end{acknowledgements}


\bibliography{references}{}
\bibliographystyle{plain}

\end{document}